\documentclass[11pt]{article}
\usepackage{amsthm, amssymb, latexsym, amsmath, color}
\usepackage[a4paper, total={6.5in, 8in}]{geometry}
\usepackage[dvipsnames,svgnames,table]{xcolor}
\usepackage[colorlinks=true,linkcolor=RoyalBlue,urlcolor=RoyalBlue,citecolor=PineGreen]{hyperref}
\usepackage{tikz}
\usepackage{tikz-cd} 

\usepackage{comment, caption}
\usepackage{graphicx,subcaption}

\usepackage{pgf,tikz,pgfplots}
\pgfplotsset{compat=1.14}
\usepackage{pgf,tikz,pgfplots}
\usepackage{mathrsfs}
\usetikzlibrary{arrows}
\usepackage{mathrsfs}
\usetikzlibrary{arrows}
\usepackage{capt-of}
\usetikzlibrary{decorations.pathreplacing}
\usetikzlibrary{cd}
\usetikzlibrary{positioning}
\usetikzlibrary{calc}
\usepackage{algorithmic}
\usepgfplotslibrary{fillbetween}
\usetikzlibrary{intersections}
\usetikzlibrary{patterns}

\usepackage[nameinlink]{cleveref}

\title{The genus of configuration curves of planar linkages \\
	is generically odd}
\author{Josef Schicho, Ayush Kumar Tewari, Audie Warren}
\newcommand{\tbd}[1]{{{\color{blue}#1}}}

\newcommand{\R}{\mathbb{R}}
\newcommand{\C}{\mathbb{C}}
\newcommand{\Q}{\mathbb{Q}}

\newcommand{\XX}{\overline{X}}
\newcommand{\YY}{\overline{Y}}

\DeclareMathOperator{\trop}{Trop}

\newtheorem{prop}{Proposition}
\newtheorem{lemma}{Lemma}

\newtheorem{theorem}{Theorem}

\newtheorem{conjecture}{Conjecture}

\theoremstyle{case}

\theoremstyle{remark}

\begin{document}

 \maketitle

\begin{abstract}
A one-degree-of-freedom graph is a graph obtained from a minimally rigid graph in the plane and removing an edge. For such graph, the set of realisations with fixed edge length, modulo rotations and reflections, is an algebraic curve. The genus of a connected component for generic edge lengths is a number that depends only on the graph. We prove that this genus is always odd, unless it is zero. The proof is based on tropical geometry.
\end{abstract}

 \section{Introduction}

Let $G=(V,E)$ be a simple graph without self-loops. If we assign generic\footnote{In this paper, a \textit{generic point} in $\mathbb C^n$ is a point whose entries are algebraically independent over $\mathbb Q$.} complex numbers to the edges of $G$, thought of as edge lengths, we can consider the set of placements of the vertices $V$ into $\mathbb C^2$ such that the squared distance of two vertices connected by an edge is equal to the assigned number. Such a placement is called a \textit{realisation} of $G$ with respect to the chosen edge lengths. Two such realisations are considered equivalent up to rotation and translation. This set of realisations is the zero set of a system of algebraic equations, and is called the \textit{generic configuration space} of $G$. A formula for the number of irreducible components of this algebraic set is known \cite{Lubbes2025}. If the dimension of the set is one, i.e. we have a {\em configuration curve}, then we say that $G$ is a \textit{one degree of freedom} (1-dof) graph; the equality $|E|=2|V|-4$ is a necessary condition. 

In \cite{genusalg}, we use tropical geometry to compute the genus of the generic configuration curve of a 1-dof graph (all irreducible components of the curve have the same genus). For example, the \emph{Jansen} mechanism \cite{sidman25} which is a 1-dof graph and appears as a leg in the famous moving sculptures \emph{Strandbeests} \cite{strandbeest}, has genus five. Upon calculating the genus for various 1-dof graphs, we (and several other colleagues) made two observations:

\begin{itemize}
    \item The genus of a 1-dof graph always seems to be either zero, or a odd number.
    \item When the genus of the 1-dof graph is zero, the graph always seems to be two rigid graphs connected at a single vertex.
\end{itemize}
In this paper, we prove that these two observations do indeed hold in general. The proof is comparatively easy in the case where $G$ has no rigid subgraphs other than edges, or equivalently, that the number of irreducible components is one. If $C$ is the configuration curve, then the set of all configurations {\em upon additionally taking the quotient by reflection} is also a curve, call it $C'$, and the quotient gives a two-to-one mapping from $C$ to $C'$. The Riemann-Hurwitz formula together with an analysis of the branch points proves the desired result.

As soon as the graph has a larger rigid subgraph, for instance a triangle, the irreducible components come in pairs related by reflections of the triangle (or whatever other rigid subgraphs the graph may have). The above map is birational on every component and does not give any information about the genus.

Our proof proceeds by introducing special instances of realisations defined by the property that points corresponding to vertices of a rigid subgraph are always collinear. This leads to the following question: {\em How do you prove that two algebraic curves have the same genus?} In this paper, we will prove this in the following way: both curves have a smooth tropicalisation, and therefore it suffices to prove that the tropical genera are equal. In fact, under certain conditions for the parameters defining the curves, the tropicalisations will actually be \textit{equal}. With this, we prove the following result.

 \begin{theorem}\label{thm:main}
    Let $G$ be a one degree of freedom graph, and let $C$ be one component of the generic configuration curve of $G$. Then either $g(C)$ is odd, or $g(C) =0$, in which case $G$ is given by two minimally rigid subgraphs $G_1$ and $G_2$ which have one vertex in common.
\end{theorem}

 \section{Preliminaries}

 \subsection{Standard definitions}

 Given a graph $G=(V,E)$, consider realisations of $G$ in $\mathbb C^2$ such that en edge $\{v_1,v_2\}$ of $G$ has its endpoints placed at $(0,0)$ and $(1,0)$. This is done to remove translations and rotations. The edge map takes as input positions for all the other vertices, and measures the squared edge lengths.
 \begin{gather*}
     f_G: \mathbb C^{2|V|-4} \rightarrow \mathbb C^{|E|-1} \\
     (x_v,y_v)_{v \in V \setminus \{v_1,v_2\}} \rightarrow ((x_{v} - x_{v'})^2 + (y_{v} - y_{v'})^2)_{\{v,v'\} \in E}    
 \end{gather*}
The graph $G$ is called \textit{dominant} if $f_G$ is a dominant morphism, that is, the image of $f_G$ is Zariski dense within $\mathbb C^{|E|-1}$. A graph is called \textit{rigid} if a generic fiber of $f_G$ is finite. The graph is called \textit{minimally rigid} if it is rigid and dominant.

Our object of study are graphs which are dominant and such that the generic fiber of $f_G$ is one dimensional. Such graphs are called 1-dof graphs, and in this case the generic fiber of $f_G$ is also called the generic configuration curve of $G$. We will not use the edge map in the way it is introduced above - this is because the map above always norms a certain edge to length one, which is a restriction we do not wish to make.

It is proved in \cite{Lubbes2025} that the generic fiber of $f_G$ may have many components - in general the number of components in the generic fiber is the product of the realisation numbers of all the maximal rigid subgraphs, and that all these components have the same genus. In \cite{Lubbes2025}, a map called the \textit{subgraph map} was used; in this paper we will use a slight variation of the subgraph map. To introduce it, we will first discuss the specific form that our realisation space will take. 

\subsection{The realisation space}\label{sec:realisationspace}
 For a connected, dominant graph $G=(V,E)$, the standard unpinned realisation space is simply $\mathbb C^{2|V|}$, that is, we place each vertex independently into $\mathbb C^2$. Our first restriction will be that we assume the first vertex of the graph is always pinned at the origin. Our second modification will be to ensure that the placement of vertices is sort of injective - that is, we consider the Zariski open subset 
$$U:= \left\{ \left(0,0,x_1,y_1,x_2,y_2,...,x_{|V|},y_{|V|}\right) \in \mathbb C^{2|V|} : \forall (j,k)\in E, \ x_j \pm i y_j\neq x_k \pm i y_k \right\},$$
 We apply an isomorphism to $U$, defined in the following way:
\begin{gather*}
    \phi: U \rightarrow (\mathbb C^*)^{2|E|} \\
    \left(0,0,x_2,y_2,...,x_{|V|},y_{|V|}\right) \mapsto \left((x_j-x_k)+i(y_j-y_k),(x_j-x_k)-i(y_j-y_k)\right)_{\substack{\{j,k\} \in E(G) \\ j<k}}.
\end{gather*}
We shall denote points in $(\mathbb C^*)^{2|E|}$ by $(u_1,v_1,...,u_{|E|},v_{|E|})$. 
The image of $U$ under $\phi$ is a linear subvariety $X \subseteq (\mathbb C^*)^{2|E|}$ given by the cycle equations
\begin{equation}\label{eqn:cycleequations}
    \forall \gamma \in \text{Cyc}(G), \ \sum_{e \in \gamma} u_e = 0 = \sum_{e \in \gamma} v_e.
\end{equation}
It is easy to see that these equations hold within the image of $U$. Using the fact that the number of independent cycles in a connected graph is given by $|E|-|V|+1$, we see that 
$$\dim(X) = 2|E|-2(|E|-|V|+1) = 2|V|-2 = \dim(U).$$
Furthermore, we can define an inverse morphism $\phi^{-1}:X \rightarrow U$. We show how to find a preimage within $U$ from a point in $X$; since $G$ is connected, there is some path $e_1,...,e_k$ from the first vertex $(0,0)$ to $(x_k,y_k)$. We then have
$$x_k = \frac{1}{2}\sum_{j=1}^k (u_{e_j} + v_{e_j}),\quad  y_k = \frac{1}{2i}\sum_{j=1}^k (u_{e_j} - v_{e_j}).$$
Therefore $\phi$ is an isomorphism. We are still not done; within $U$ we would have liked to identify two realisations which are rotations of each another as equivalent - this is easy to do within $X$. It is easy to check that a rotation of a realisation (around the origin) inside $U$ corresponds to a dilation
$$(u_e,v_e)_{e \in E} \mapsto \theta \cdot (u_e,v_e) := (\theta u_e, \theta^{-1} v_e)_{e\in E}, \text{ for some } \theta \in \mathbb C^*.$$
We define an equivalence relation $\sim$ on $(\mathbb C^*)^{2|E|}$, where two vectors $p_1 \sim p_2$ if there is a $\theta \in \mathbb C^*$ such that $p_2 = \theta \cdot p_1$. We can then define the \textit{realisation variety} of $G$, given by
$$R_G := \left\{ (u_1,u_2,...,u_{|E|},v_1,v_2,...,v_{|E|}) \in (\mathbb C^*)^{2|E|}/\sim \ : \text{Equations \eqref{eqn:cycleequations} hold} \right\}.$$
Note that this is well defined since validity of any of the linear equations \eqref{eqn:cycleequations} does not depend on the choince of representative from the equivalence class. 


The above definition of the $R_G$ is not economic in terms of variables. We will also use a description that uses a minimal subset
of variables. To this purpose, we choose an edge $e_0\in E$ and a spanning tree $T\subset E$ that contains $e_0$. If $e\notin T$, then the value
of $u_e$ can be uniquely recovered from the value of all $u_{e'}$ with $e'\in T$: if we add $e$ to the tree, then it will close a unique cycle,
and $u_e$ can be obtained from the corresponding cycle equations~(\ref{eqn:cycleequations}). And then we can choose a unique representive
such that $u_{e_0}=1$. The number of variables we have is then $2|T|-1=2|V|-3=\dim(R_G)$. In other words, we may identify $R_G$
with a dense open subset of $\C^{2|V|-3}$ defined by inequalities that come from the restrictions we imposed in the beginning of this subsection.

\subsection{The subgraph map}

We can now define the subgraph map. The motivation for this map is that it gives a useful factorisation of the edge map of a graph; in fact, the subgraph map can be used to `single out' one irreducible component of the edge map. The concept of the subgraph map is simple - it takes as an input a realisation of a connected dominant graph $G$, and outputs the induced realisations for the maximal minimally rigid subgraphs $G_1,...,G_r$ of $G$, where we define $r$ to be the number of maximal minimally rigid subgraphs of $G$. A maximal minimally rigid (mmr) subgraph of $G$ is a minimally rigid subgraph of $G$ which cannot be extended to any larger minimally rigid subgraph. 

Every edge of $G$ belongs to exactly one mmr subgraph (maybe consisting of the edge itself). For each mmr subgraph $G_i = (V_i,E_i)$ we select any edge, say $e_i \in E_i$. Further, for any $\rho \in R_G$, we use $\rho_{G_i}$ to denote the projection of $\rho$ to the entries with edges within $E_i$ - this is the induced realisation of $G_i$. We then define
\begin{gather*}
    S_G:R_G \rightarrow \prod_{i=1}^r R_{G_i}, \ 
    \rho 
  \mapsto (\rho_{G_i})_{i=1}^r.
\end{gather*}
 The edge map $f_G:R_G \rightarrow \mathbb C^{|E|}$ maps a class represented by $(u_1,...,u_{|E|},v_1,\dots,v_{|E|})$
to the vector of products  $(u_ev_e)_{e\in E}$ -- clearly, this does not depend on the choice of representative.
The edge map now factorises in the following way.
        \begin{equation}
        \label{diag:subgraphfactor}
        \begin{tikzcd}
        R_G \arrow[dr, "f_G"] \arrow[r, "S_G"]& \prod_{i=1}^r R_{G_i} \arrow[d,"\prod_i f_{G_i}"]\\
        & \mathbb C^{|E|}
        \end{tikzcd}
        \end{equation}
It was proven in \cite{Lubbes2025} that $S_G$ is dominant, and its generic fiber is geometrically irreducible. The dimension of the generic fiber is the number of degrees of freedom of the graph $G$, that is, $\dim(S_G^{-1}(\rho_1,...,\rho_r)) = 2|V|-|E|-3$ for generic $\rho_1,...,\rho_r$. This is true since the product $\prod_i f_{G_i}$ is a generically finite map, because each $G_i$ is a minimally rigid graph. By the factorisation above, we can work directly with the subgraph map instead of the edge map. From hereon, the graph $G$ is always assumed to be a one-degree-of-freedom graph. 

\section{Tropicalising the fibers of $S_G$}

In this section we will tropicalise certain fibers of $S_G$. In order to do this, we must consider these fibers as being defined over the field of Hahn series $\mathbb C\{\{t\}\}$, with solutions lying within the algebraic torus $(\mathbb C\{\{t\}\}^*)^{2|E|}$. A Hahn series is a series of the form
$$p(t) =\sum_{i=k}^{+\infty} c_i t^i,$$
where the sum is over a well -ordered subset of $\R$ which is bounded below and has no acumulation points, $k$ is the minimum of this set, and each $c_i$ is a complex number. The valuation of such a series is 
$$v(p(t)) = k.$$
For background on tropical geometry we recommend the book of Maclagan and Sturmfels \cite{maclagansturmfels}. From now on all parameters and variables are considered as nonzero Hahn series.

\subsection{Fibers of $S_G$}

Our method of proof is to compare the tropicalisations of two specific fibers of the subgraph map $S_G$. One of these fibers will act as a generic (albeit specifically chosen!) fiber, and the other will be a `special' fiber where all realisations $\rho_1,...,\rho_r$ of the maximal minimally rigid subgraphs $G_1,...,G_r$ are collinear.

Let us start be writing down the algebraic equations defining a fiber. It will be useful to use all coordinates of $R_G$, but only a minimal subset of coordinates for each of the $R_{G_i}$. To make this more precise, we introduce two simple functions. The first we call $n$, and is a map from $E$ to the set of maximal rigid subgraphs of $G$, which maps an edge to the rigid subgraph it is contained in.
$$n:E \rightarrow \{G_1,...,G_r\}, \quad n(e) = G_i \ \text{ such that }e \in G_i.$$

Secondly, we define a map $s:\{G_1,...,G_r\} \rightarrow E$ which selects a single edge from each subgraph. The map $s$ is a right inverse of $n$.
$$s:\{G_1,...,G_r\} \rightarrow E, \ \forall e \in E \  n(s(e)) = e.$$
Finally, we will choose a spanning tree $T_i$ for each $G_i$, such that the special edge $s(G_i)$ is contained in the tree $T_i$ for each $i$. 
We use $T$ to denote the graph-theoretic union of all these trees.

The image point is specified by a vector in 
$$(\underline{\lambda}, \underline{\mu}, \underline{c}) \in  \prod_{i=1}^r(\mathbb C\{\{t\}\}^*)^{|T_i|-1} \times  \prod_{i=1}^r(\mathbb C\{\{t\}\}^*)^{|T_i|-1} \times  (\mathbb C\{\{t\}\}^*)^r.$$ 
For each subgraph $G_i$, the image realisation is given by a representative such that the $u$-coordinate of the edge $s(G_i)$ is one. All
$v$-coordinates of edges in $s(G_i)$ are collected in $\underline{c}\in (\mathbb C\{\{t\}\}^*)^r$. And we just specify the values of variables
corresponding to edges in the tree -- see the last paragraph of Section~\ref{sec:realisationspace}.

The fibers of $S_G$ can now be written algebraically in the following way. We have the given parameters
The equations defining $S_G^{-1}(\underline{\lambda}, \underline{\mu}, \underline{c})$ come in four groups. 
\begin{gather}
    u_e = u_{s(G_i)} \lambda_e, \ \forall i =1,...,r, \ e \in E_i \label{lambdaeq} \\
    v_e  = v_{s(G_i)} \mu_e,  \ \forall i =1,...,r, \ e \in E_i \label{mueq} \\ 
    u_{s(G_i)}v_{s(G_i)} = c_i, \forall i=1,...,r \label{scaleeq} \\
    \sum_{e \in \gamma} u_e = 0 = \sum_{e \in \gamma} v_e, \ \forall \gamma \in \text{Cyc}(G),\label{cycleeq}
\end{gather}
where in these equations any appearance of $\lambda_e$ with $e \notin T_i$ is given by the cycle equations, and with $\mu_{s(G_i)} = \lambda_{s(G_i)} = 1$. The equations \eqref{lambdaeq} and \eqref{mueq} enforce that the induced realisations of each $G_i$ is similar (meaning up to translation, rotation, and dilation) to the realisations given by $(\underline{\lambda}, \underline{\mu})$. The equations $\eqref{scaleeq}$ enforce a scaling $c_i$ onto each rigid subgraph $G_i$. Finally, we need to satisfy the cycle equations \eqref{cycleeq}. 
Given this specific form of the fiber, we prove the following lemma concerning a certain special fiber of $S_G$.

\begin{lemma}\label{lem:collinearfiber}
    For generic choice of $(\underline{\lambda},\underline{c}) \in \prod_{i=1}^r(\mathbb C\{\{t\}\}^*)^{|T_i|-1} \times  (\mathbb C\{\{t\}\}^*)^r$, the fiber $S_G^{-1}(\underline{\lambda},\underline{\lambda},\underline{c})$ consists of realisations of $G$ such that each induced realisation on $G_i$ is collinear.
\end{lemma}

\begin{proof}
    Consider the realisation variety $R_{G_i}$ for any $i$. In the original coordinate system, the realisation of $G_i$ being collinear means that the vertex vector $(x_i, y_i)_{i=1}^{|V|}$ is of the form $y_i = \alpha x_i + \beta$ for some fixed $\alpha,\beta \in \mathbb C$. Under the isomorphism $\phi$, the $u,v$ coordinates in $R_{G_i}$ for such a realisation are of the form $(u_e, \left(\frac{1 + \alpha i}{1 - \alpha i}\right) u_e)$. Note that the value $\beta$ cancels through the isomorphism. Now consider the fiber $S_G^{-1}(\underline{\lambda},\underline{\lambda},\underline{c})$. For each rigid subgraph $G_i$, an element of this fiber must induce a realisation of the form $(\lambda_e, c_i \lambda_e)_{e \in E_i}$ on $G_i$. We see that this means that $G_i$ is being realised on a line, by setting $c_i = \frac{1 + \alpha i}{1 - \alpha i}$ as above. This finishes the proof since each rigid subgraph is collinear.
\end{proof}

\subsection{Two specific fibers of $S_G$}

Recall that a vector in a complex vector space is generic iff its coordinates are algebraically independent over $\Q$.
Similarily, we say that a vector over the field of Hahn series has {\em generic valuation} if its entries are not zero and the vector of the valuation of its entries are linearly independent over $\Q$. This implies that the coordinates are algebraically independent over $\Q$ (even over $\C$).

The tropical curves in \cite{genusalg} are tropicalisations of fibers of the subgraph map whose image points have generic valuation. In this paper, we choose two fibers in another way. We fix a generic vector $\underline{\tau} \in \prod_{i=1}^r(\mathbb C^*)^{|T_i|-1}$ and consider its coordinates as Hahn series with valuation zero. We also fix a vector $\underline{c} = (c_1,...,c_{r})\in(\mathbb C\{\{t\}\}^*)^{r}$ with generic valuation. 
We then define the fibers
$$F_1 := S_G^{-1}(\underline{\tau},\underline{\tau},\underline{c}). $$
Then we fix another generic vector $\underline{\delta} \in \prod_{i=1}^r(\mathbb C^*)^{|T_i|-1}$ and consider its coordinates as Hahn series with valuation zero. And we define
$$F_2:= S_G^{-1}(\underline{\tau},\underline{\delta},\underline{c}).$$

Given these two fibers, we wish to prove the following result. Note that a tropical curve is called smooth if it at every point it is locally tropically isomorphic to a tropical line, see \cite{genusalg} for more details.

\begin{prop}\label{prop:equalandsmooth}
    The two tropical curves $\trop(F_1)$ and $\trop(F_2)$ are equal. Furthermore, this tropical curve is smooth.
\end{prop}

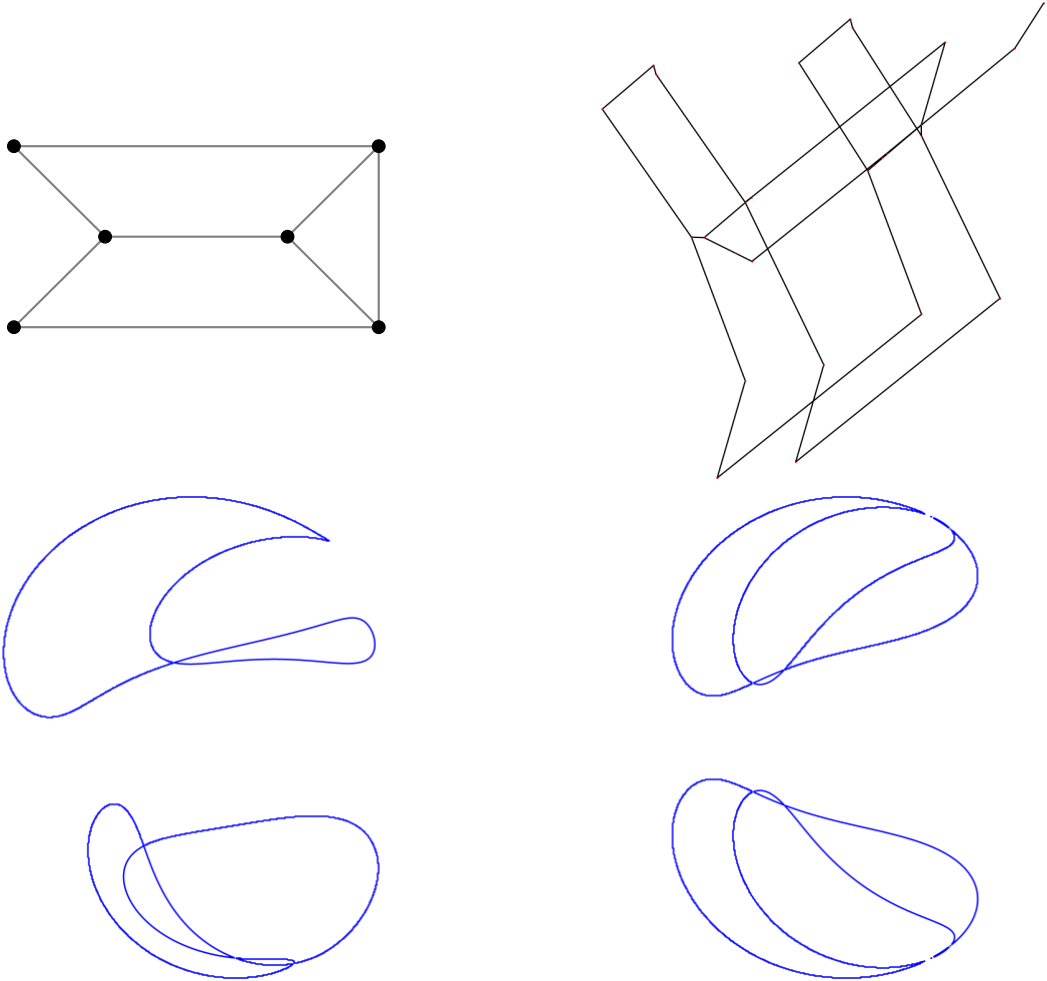
\begin{figure}[h!]
\begin{minipage}{0.5\textwidth}
\centering
     \begin{tikzpicture}[scale =1.2]
        \draw[gray, thick] (-1,0) -- (1,0);
        \draw[gray, thick] (-1,0) -- (-2,-1);
        \draw[gray, thick] (-2,1) -- (-1,0);
        \draw[gray, thick] (1,0) -- (2,-1);
        \draw[gray, thick] (2,1) -- (2,-1);
        \draw[gray, thick] (-2,1) -- (2,1);
        \draw[gray, thick] (-2,-1) -- (2,-1);
        \draw[gray, thick] (2,1) -- (1,0);
        \filldraw[black] (-1,0) circle (2pt);
        \filldraw[black] (1,0) circle (2pt);
        \filldraw[black] (-2,-1) circle (2pt);
        \filldraw[black] (-2,1) circle (2pt);
        \filldraw[black] (2,-1) circle (2pt);
        \filldraw[black] (2,1) circle (2pt);
    \end{tikzpicture}
    
\end{minipage}
\begin{minipage}{0.5\textwidth}
\centering
    \includegraphics[width=0.8\linewidth]{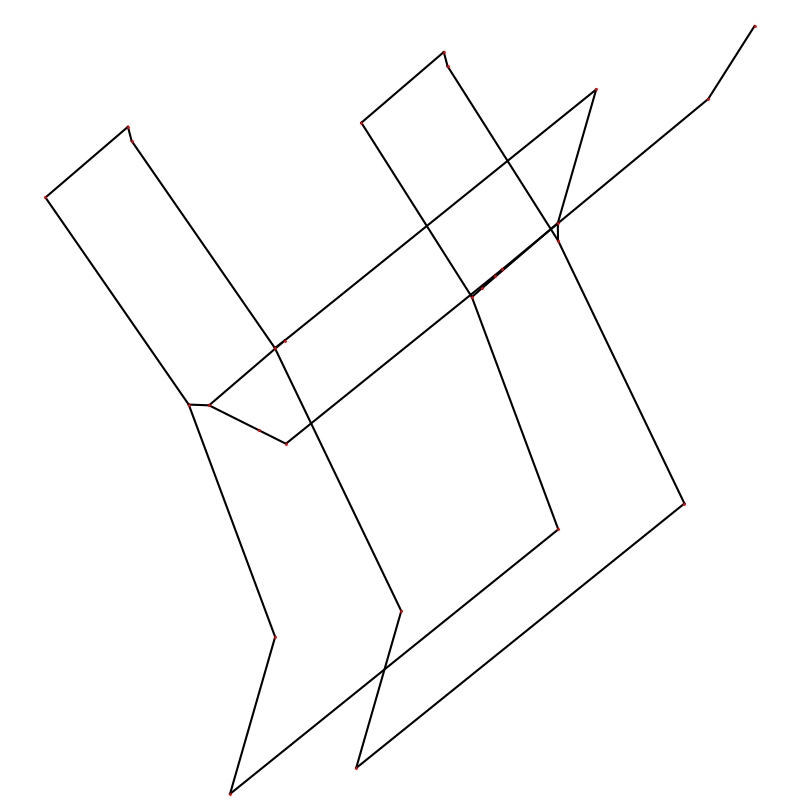}
    \label{fig:placeholder}
\end{minipage}

    \begin{minipage}{0.5\textwidth}
\centering
    \includegraphics[width=0.8\linewidth]{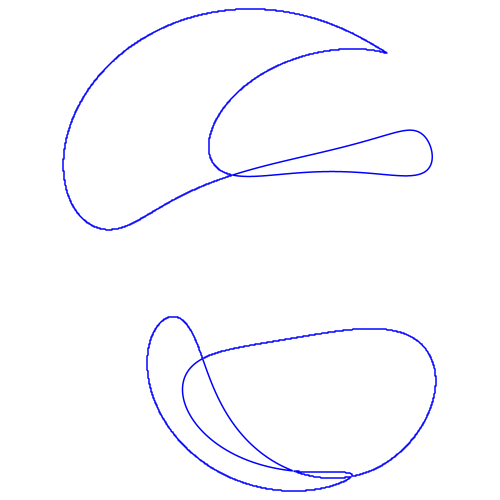}
    \label{fig:placeholder}
\end{minipage}
\begin{minipage}{0.5\textwidth}
\centering
    \includegraphics[width=0.8\linewidth]{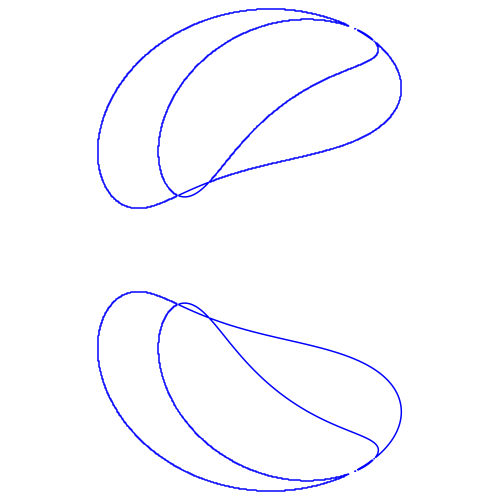}
    \label{fig:placeholder}
\end{minipage}
    \caption{Top left: A graph $G$ of genus five. Top right: The tropicalisation of its generic fiber. Bottom left: A projected image of the algebraic generic fiber $F_1$ (two real components). Bottom right: A projected image of the algebraic special fiber $F_2$ (two real components), with a visible symmetry. (The tropicalisation is also symmetric, but the symmetry is not visible in the projection.)}   
    \label{fig:placeholder}
\end{figure}

\subsection{The tropical varieties $\XX$ and $\YY$}

We follow a process from \cite{genusalg} to calculate the tropicalisations of $F_1$ and $F_2$. That is, we write these fibers as the intersection of two varieties. To do this, we return to the equations \eqref{lambdaeq}, \eqref{mueq}, \eqref{scaleeq}, and \eqref{cycleeq}. Using equations \eqref{lambdaeq} and \eqref{mueq}, we can calculate any variable $u_e$ (resp. $v_e$) by knowing just the value $u_{s(n(e))}$ (resp. $v_{s(n(e))}$). This allows us to eliminate down to just the variables $(u_{s(G_i)},v_{s(G_i)})_{i=1}^r$. Furthermore, we can use the equivalence relation in $R_G$ to pick representatives with $u_{s(G_1)}=1$. Finally, we can eliminate all of the $v_{s(G_i)}$ variables using equations \eqref{scaleeq}. What we are left with is a curve within $(\mathbb C\{\{t\}\}^*)^{r-1}$ which is given by the two sets of equations
\begin{gather}
    \sum_{e \in \gamma} u_{s(n(e))} \lambda_e = 0 \label{eqn:X}, \quad \gamma \in \text{Cyc}(G)\\
    \sum_{e\in \gamma} \frac{c_{n(e)}}{u_{s(n(e))}}\mu_e = 0, \quad \gamma \in \text{Cyc}(G). \label{eqn:Y}
\end{gather}
we have abused notation slightly here, using $c_{n(e)}$ to denote $c_i$ where $n(e) = G_i$. We further note that we have left $u_{s(G_1)}$ in these equations for simplicity, but as we have seen earlier, this variable is equal to $1$. Note that all these reductions we have made stay valid when applied to the specific fibers $F_1$ and $F_2$. We now make four definitions, referring to the fixed parameters $\underline{\tau}, \underline{\delta}$, and $\underline{c}$ from the definitions of $F_1$ and $F_2$.
\begin{gather*}X_{\text{lin}} := V\left( \sum_{e \in \gamma} u_{s(n(e))} \tau_e,\ \gamma \in \text{Cyc}(G)\right) \subseteq (\mathbb C\{\{t\}\}^*)^{r-1} \\
Y_{\text{lin}} := V\left( \sum_{e \in \gamma} \frac{c_{n(e)}}{u_{s(n(e))}}\tau_e, \ \gamma \in \text{Cyc}(G)\right) \subseteq (\mathbb C\{\{t\}\}^*)^{r-1} \\
X_{\text{gen}} := V\left( \sum_{e \in \gamma} u_{s(n(e))} \tau_e,  \ \gamma \in \text{Cyc}(G)\right) \subseteq (\mathbb C\{\{t\}\}^*)^{r-1}  \\
Y_{\text{gen}} := V\left( \sum_{e \in \gamma}  \frac{c_{n(e)}}{u_{s(n(e))}}\delta_e,\ \gamma \in \text{Cyc}(G)\right) \subseteq (\mathbb C\{\{t\}\}^*)^{r-1} 
\end{gather*}
We then have that $F_1 = X_{\text{lin}} \cap Y_{\text{lin}}$ and $F_2 = X_{\text{gen}} \cap Y_{\text{gen}}$. 
Note that $X_{\text{lin}}$ and $X_{\text{gen}}$ are equal - we distinguish them for exposition. 
Note also that $Y_{\text{lin}}$ and $Y_{\text{gen}}$ only differ by a different choice of generic complex parameters. The bijection from $Y_{\text{lin}}$ to $Y_{\text{gen}}$ induced by mapping $\tau_e$ to $\delta_e$ preserves the valuation, and therefore $Y_{\text{lin}}$ and $Y_{\text{gen}}$ have the same tropicalisation. 

We define
\[ \XX = \trop(X_{\text{lin}})=\trop(X_{\text{gen}}),\  
\YY = \trop(Y_{\text{lin}})=\trop(Y_{\text{gen}}) . \]
Both are smooth tropical varieties of dimension $\frac{r}{2}$ in $\R^{r-1}$.

The following lemma is an analogue of \cite[Lemma 2]{genusalg}.

\begin{lemma} \label{thm:transversal}
$\XX$ and $\YY$ intersect transversally at any of its intersection points.
\end{lemma}

\begin{proof}
 Let $w=(v(c_1),\dots,v(c_r))\in (\C^\ast)^r$. Then $\YY=w-\XX$. 

 Let $p\in \XX\cap\YY$. Let $f_X$ be a face of $X$ and let $f_Y$ be a face of $\YY$, both containing $p$.
  Let $L_X$ be the affine span of $f_X$ and let $L_Y$ be the affine span of $f_Y$. 
  Then $L_X$ contains zero, because zero is in the closure of every face of $X$. The translate
  $L_Y-w$ also contains zero, by the same reason. So, both $L_\sigma$ and $L_\tau-w$ are
  vector spaces. Moreover, both vector spaces are defined over ${\mathbb Q}$. Let $W:=L_X+(L_Y-w)$.
  Since $p\in L_X$ and $-p\in L_Y$, it follows that $w\in W$. But $w$ is generic, and
  $W$ is defined over ${\mathbb Q}$. It follows that $W={\mathbb R}^{r-1}$, which means $\XX$ and
  $\YY$ intersect transversally.   
\end{proof}

As a consequence of transversality, it follows from \cite[3.4.12]{maclagansturmfels} that the tropicalization of the generic fibers are both equal to the intersection of the
two tropical varieties.

\begin{proof}[Proof of Proposition~\ref{prop:equalandsmooth}]
By \cite[Theorem~3.4.14]{maclagansturmfels}, intersection commutes with tropicalisation in the case the intersection of the tropical varieties is transversal.
Hence
\[ \trop(F_1)=\trop(X_{\text{lin}} \cap Y_{\text{lin}}) = \trop(X_{\text{lin}})\cap \trop(Y_{\text{lin}}) = \XX\cap\YY \]
\[ = \trop(X_{\text{gen}})\cap \trop(Y_{\text{gen}}) = \trop(X_{\text{gen}} \cap Y_{\text{gen}}) = \trop(F_2) , \]
which is the first claim of the proposition.

The second claim, namely the smoothness of the intersection, can be proven exactly as the second part of
\cite[Lemma 3]{genusalg}: each point $p$ of the intersection of a maximal face of $\XX$ and a codimension~1 face of $\YY$,
or conversely. So, in some neighborhood $U$ of $p$, we may replace $\XX$ or $\YY$ by its point reflection $2p-\XX$
or $2p-\YY$. Note that $2p-\YY$ is the tropicalisation of a linear variety. If $p$ is in a maximal face of $\YY$,
then inside $U$, the intersection coincides with the tropicalisation of the intersection of two linear varieties. 
So, it is the tropicalisation of a linear variety, and therefore smooth. The second case is similar, as there
is a symmetry of the roles of $\XX$ and $\YY$.
\end{proof}

\subsection{Irreducibility of $F_1$}

The last task of this section is to prove that the fiber $F_1$ is irreducible. 

\begin{lemma} \label{lem:F1isirreducible}
    The fiber $F_1$ is an irreducible curve.
\end{lemma}

Note that we already know from \cite{Lubbes2025} that $F_2$ is irreducible, however this only applies to \textit{generic} fibers, and $F_1$ is not generic since it consists of realisations with each mmr subgraph $G_i$ collinear. Nevertheless we will show that $F_1$ is indeed irreducible, since we know its tropicalisation is smooth and connected. Smoothness we have shown above, and connectedness follows from $\trop(F_1) = \trop(F_2)$ and the fact that $\trop(F_2)$ is connected since $F_2$ is irreducible. We will therefore need the following result.

\begin{lemma}\label{lem:connectedmeansirreducible}
    Let $C$ be a smooth algebraic curve. If $\trop(C)$ is smooth and connected, then $C$ is irreducible.
\end{lemma}
We remark that actually something strictly stronger is true; the assumption of smoothness of $C$ is actually already a consequence of the smoothness of $\trop(C)$ \cite[Section 5]{jell2020constructing}. However, we do not need this here.

\begin{proof}[Proof of \Cref{lem:connectedmeansirreducible}]
    We will use contradiction. Assume that $C$ is reducible. As $\trop(C)$ is connected, there exist two components $C_1,C_2$ of $C$ such that $\trop(C_1)$ and $\trop(C_2)$ intersect. If this intersection is transverse, then we can use \cite[Theorem~3.4.13]{maclagansturmfels} to conclude that 
    $$\trop(C_1) \cap \trop(C_2) = \trop(C_1 \cap C_2) = \emptyset,$$
    which follows since $C$ is smooth so its components cannot intersect. This is a contradiction since we assume that the tropicalisations \textit{do} intersect. Therefore the intersection must be non-transverse. This non-transverse intersection can occur in various ways.

    \begin{enumerate}
        \item Two 1-d cells intersect in a single point, call it $p$. If this occurs then the intersection point is not smooth as there are edges leaving $p$ in opposite directions, and so cannot be locally isomorphic to a tropical line at $p$.
        \item Two 1-d cells intersect in a 1-d segment. If this occurs then the 1-d segment has weight greater than one in the tropical curve, and so cannot be smooth.
        \item A 1-d cell intersects a 0-d cell. Just as in the first case, the 0-d cell has two edges leaving it in opposite directions, which cannot be smooth.
        \item Two 0-d cells intersect. In this case we will use the following small lemma.
    \end{enumerate}
\begin{lemma}
 Let $T$ be a smooth tropical curve, and let $v$ be a vertex (0-dimensional cell) of $T$. Then no strict subset of the 1-d cells leaving $v$ can satisfy the balancing condition.
\end{lemma}

\begin{proof}
    Since $T$ is smooth, there must be a tropical isomorphism $M$ from a small open set $U_v$ containing $v$, to some tropical line. If we suppose that $v$ has valency $k$, then $M$ can be chosen to map $U_v$ so that $M(v)= \underline{0} \in \mathbb R^{k-1}$, and so that the cells leaving $v$ have direction vectors $\left(e_1,e_2,...,e_{k-1}, -\sum_{i=1}^
    {k-1} e_i\right)$, where each $e_i$ is the unit vector in coordinate direction $i$. The balancing condition is preserved by tropical isomorphisms, and it is clear that no strict subset of the direction vectors from $M(U_v)$ satisfies the balancing condition. This concludes the proof.
\end{proof}

In fact, the first and third cases above are also special cases of this lemma. For the fourth case, if the two 0-d cells intersect (and assuming the 1-d cells leaving this cell are distinct, since otherwise we have case two from above), the subset of 1-d cells from, say, $\trop(C_1)$ satisfy the balancing condition, and this contradicts the lemma above. Therefore the tropical curve is not smooth, finishing the proof of \Cref{lem:connectedmeansirreducible}. 
\end{proof}

Since $\trop(F_1)$ is smooth and connected, \Cref{lem:connectedmeansirreducible} implies that $F_1$ is irreducible, proving \Cref{lem:F1isirreducible}.

\section{Applying the Riemann-Hurwitz Theorem}

In this section we will apply the Riemann-Hurwitz theorem to prove \Cref{thm:main}. We recall that the two fibers we are interested in are $F_1$, which is a fiber of the subgraph map which gives each mmr subgraph of $G$ a collinear realisation, and $F_2$, which is a specifically chosen but algebraically fully generic. Since we know that the two tropical curves $\trop(F_1)$ and $\trop(F_2)$ are smooth, by \cite[Theorem 1]{genusalg} the (geometric) genus of the curves $F_1$, $F_2$, is equal to the genus of the tropical curves $\trop(F_1)$ and $\trop(F_2)$. The genus of a tropical curve is the number of bounded edges minus the number of vertices, plus one - that is, it is the genus of the underlying graph given by removing all infinite rays. However, we have already proved that these two tropical curves are \textit{equal}, so that $g(F_1) = g(F_2)$. We now aim to give a formula for $g(F_1)$.

Consider the morphism $\phi$ which maps each element $\rho$ of $F_1 \subseteq R_G$ to its equivalence class $[\rho]$ given by reflection. To be explicit, the reflection map is an automorphism of $R_G$ which swaps the $u$ and $v$ variables, that is,
$$(u_e,v_e)_{e \in E} \mapsto (v_e,u_e)_{e \in E}.$$
The `quotient morphism' is then given by
$$\phi: F_1 \rightarrow C, \quad \phi(\rho) = [\rho],$$
where $C$ is some one-dimensional variety - it will not be relevant exactly what this image curve is. The morphism $\phi$ is \tbd{surjective} (by definition) and is a $2:1$ map, since a generic class has two elements. Applying the Riemann-Hurwitz formula to this map yields
$$2g(F_1) - 2 = 2(2g(C)-2) + \#\{\text{Ramification points}\}.$$
What are the ramification points of this map? Since they must be points in $F_1$ with only a single pre-image, they correspond to \textit{completely collinear} realisations of $G$ in $F_1$. We will now prove the following lemma.

\begin{lemma}\label{lem:branchpoints}
  The map $\phi$ only has branch points when $G$ has only two rigid components, that is, when $r=2$.  
\end{lemma}

\begin{proof}
    For each mmr subgraph $G_i$, define $L_{G_i}$ to be the variety of realisations of $G_i$ such that every vertex is on a common line. Let $L_G$ be the variety of realisations of $G$ such that each mmr subgraph is collinear - this is precisely the pre-image under the subgraph map $S_G$ of the variety
    $$\prod_{i=1}^r L_{G_i} \subseteq \prod_{i=1}^r R_{G_i},$$
    and since fibers of elements of $\prod_{i=1}^r L_{G_i}$ are generically one-dimensional (which follows from $\trop(F_1)$ being a smooth tropical curve), we have
    \begin{align*}
        \dim(L_{G}) &= 1 + \sum_{i=1}^r\dim(L_{G_i}) \\
        & = 1 + \sum_{i=1}^r (|V_i|-1) \\
        & = 1 + \sum_{i=1}^r (\frac{|E_i| + 3}{2}) \\
        & = \frac{|E| + r}{2} + 1.
    \end{align*}
Secondly, we define $CL_G$ to be the variety of realisations of $G$ with \textit{every point} on a common line. Clearly $CL_G \cong \mathbb C^{|V|-1}$, and so we have
$$\text{codim}_{L_G}(CL_G) = \frac{r}{2}.$$
Points of $CL_G$ correspond to ramification points of $\phi$. On the other hand, a generic fiber of
$$S_G : L_G \rightarrow \prod_{i=1}^r L_{G_i}$$
only intersects $CL_G$ if $\text{codim}_{L_G}(CL_G) = 1$. Therefore branch points only occur if  $r=2$.
\end{proof}
We can now finish the proof of \Cref{thm:main}. Indeed, firstly suppose that $g(F_1) = 0$. Since $\phi$ is a morphism, we must also have $g(C) =0$, and so Riemann-Hurwitz yields
$$\#\{\text{Ramification points}\} = 2,$$
and so since there are ramification points we have $r=2$ by \Cref{lem:branchpoints}, so $G$ is given by two rigid subgraphs connected at a single vertex. Now assume that $g(F_1) \neq 0$. We then have $r > 2$, since if $r=2$ the fiber is isomorphic to a circle and so has genus zero. But then there are no ramification points, and so by the Riemann-Hurwitz formula
$$g(F_1) = 2g(C)-1,$$
and so $g(F_1)$ is odd. This concludes the proof of \Cref{thm:main}.

\section{Further remarks}

In this paper we classified graphs with generic fiber having genus zero; the next natural step would be to classify those graphs whose fiber has genus one. From small experiments, it seems likely that those graphs $G$ with genus one are equivalent to a four-cycle, in the the sense that $G$ has exactly four maximal rigid graphs which must be attached in a cycle. Note that there is no other way to connect four maximal rigid subgraphs in a 1-dof graph.

\begin{conjecture}
Let $G$ be a 1-dof graph whose generic fiber has genus one. Then $G$ has four maximal rigid subgraphs, and is therefore `equivalent to' a four-cycle.
\end{conjecture}

Furthermore, in this paper we proved that the genus of a 1-dof graph must be odd (if non-zero), however it is not clear which odd numbers may appear. It is not even clear whether a graph with genus three exists. By running the algorithm from \cite{genusalg} on graphs with at most nine vertices, we have examples with the following genera: 
$$0,1,5,7,17,21,23,25,27,33,49,55,57,65,69,73,129,145,151,321.$$
We thank Jose Capco for running these experiments. An over-arching problem would be to classifiy exactly which odd numbers can appear, although this seems to be a very difficult problem and we have no conjecture as to the pattern.

\section*{Acknowledgments}

A.W. was partially supported by Austrian Science Fund (FWF): 10.55776/PAT2559123.

 \vspace{4mm}
\noindent Research Institute for Symbolic Computation (RISC)
\\Johannes Kepler University, Linz, Austria
\\\url{josef.schicho(at)risc.jku.at}
\\[4mm]
Johann Radon Institute for Computational and Applied Mathematics (RICAM)
\\Austrian Academy of Sciences, Linz, Austria
\\\url{audie.warren(at)oeaw.ac.at}
\\\url{ayushkumar.tewari(at)oeaw.ac.at}

\bibliography{oddgenus}
\bibliographystyle{plain}

\end{document}